\documentclass[a4paper,leqno]{fic-l}
\usepackage{amsmath,amsthm,amssymb,amsfonts}
\usepackage{enumerate,amscd,amsxtra}
\usepackage{mathrsfs}
\usepackage[cmtip,all]{xy}

\normalsize

\newtheoremstyle{theoremstyle}
  {10pt}      
  {5pt}       
  {\itshape}  
  {}          
  {\bfseries} 
  {:}         
  {.5em}      
  {}          

\newtheoremstyle{examplestyle}
  {10pt}      
  {5pt}       
  {}          
  {}          
  {\bfseries} 
  {:}         
  {.5em}      
  {}          

\theoremstyle{theoremstyle}
\newtheorem{theorem}{Theorem}[section]
\newtheorem*{theorem*}{Theorem}
\newtheorem{lemma}[theorem]{Lemma}
\newtheorem{proposition}[theorem]{Proposition}
\newtheorem*{proposition*}{Proposition}
\newtheorem{corollary}[theorem]{Corollary}
\newtheorem*{corollary*}{Corollary}

\newtheorem{example}[theorem]{Example}
\newtheorem{definition}[theorem]{Definition}
\newtheorem{definition*}{Definition}
\newtheorem{remark}[theorem]{Remark}
\newtheorem{remark*}{Remark}

\newcommand{\x}{\mathsf{x}}
\newcommand{\y}{\mathsf{y}}
\newcommand{\MU}{\mathsf{MU}}
\newcommand{\KU}{\mathsf{KU}}
\newcommand{\Z}{\mathbf{Z}}
\newcommand{\CP}{\mathbf{CP}}
\newcommand{\LK}{\mathsf{LK}}
\newcommand{\MGL}{\mathsf{MGL}}
\newcommand{\KGL}{\mathsf{KGL}}
\newcommand{\EE}{\mathsf{E}}
\newcommand{\BGL}{\mathsf{B}\mathbf{GL}}
\newcommand{\BU}{\mathsf{BU}}
\renewcommand{\AA}{{\mathbf A}}
\newcommand{\PP}{{\mathbf P}}
\newcommand{\TT}{{\mathbf T}}
\newcommand{\thclass}{\mathsf{th}}
\newcommand{\chclass}{\mathsf{ch}}
\newcommand{\Th}{\mathsf{Th}}
\newcommand{\GGG}{\mathsf{G}}
\newcommand{\caT}{{\mathcal T}}
\newcommand{\caL}{{\mathcal L}}
\newcommand{\op}{\mathrm{op}}
\newcommand{\MMM}{\mathsf{M}}
\newcommand{\Hom}{\mathsf{Hom}}
\newcommand{\cc}{\mathsf{c}}
\newcommand{\Gr}{\mathbf{Gr}}
\newcommand{\Pic}{\mathsf{Pic}}
\newcommand{\Set}{\mathsf{Set}}
\newcommand{\uPic}{\underline{\mathsf{Pic}}}
\newcommand{\A}{\mathbf{A}}
\renewcommand{\H}{\mathbf{H}}
\newcommand{\G}{\mathbf{G}_{\mathfrak{m}}}
\newcommand{\SH}{\mathbf{SH}}
\newcommand{\BG}{\mathsf{B}\mathbf{G}_{\mathfrak{m}}}
\newcommand{\LQ}{\mathsf{L}\mathbf{Q}}
\newcommand{\HQ}{\mathsf{H}\mathbf{Q}}
\newcommand{\id}{\mathsf{id}}
\newcommand{\Sm}{\mathsf{Sm}}
\newcommand{\Spec}{\mathsf{Spec}}
\newcommand{\Ho}{\mathsf{Ho}}
\newcommand{\Nis}{\mathsf{Nis}}
\newcommand{\Zar}{\mathsf{Zar}}
\newcommand{\colim}{\mathrm{colim}}
\newcommand{\sPr}{\mathbf{sPre}}

\title{\bf Chern classes, K-theory and Landweber exactness over nonregular base schemes}
\author{Niko Naumann, Markus Spitzweck, Paul Arne {\O}stv{\ae}r}
\date{\today}
\begin{document}
\begin{abstract}
The purpose of this paper is twofold.
First, 
we use the motivic Landweber exact functor theorem to deduce that the Bott inverted infinite projective 
space is homotopy algebraic $K$-theory.
The argument is considerably shorther than any other known proofs and serves well as an illustration of 
the effectiveness of Landweber exactness.
Second,  
we dispense with the regularity assumption on the base scheme which is often implicitly required in the 
notion of oriented motivic ring spectra. 
The latter allows us to verify the motivic Landweber exact functor theorem and the universal property of 
the algebraic cobordism spectrum for every noetherian base scheme of finite Krull dimension. 
\end{abstract}
\maketitle
\tableofcontents
\newpage

\section{Landweber exactness and  $K$-theory}
\label{section:LEandhaKtheory}
It has been known since time immemorial that every oriented ring spectrum acquires a formal group law \cite{Adams}.
Quillen \cite{Quillen} showed that the formal group law associated to the complex cobordism spectrum $\MU$ 
is isomorphic to Lazard's universal formal group law;
in particular, 
the Lazard ring that corepresents formal group laws is isomorphic to the complex cobordism ring $\MU_{\ast}$.
This makes an important link between formal group laws and the algebraic aspects of stable homotopy theory.
Let $\KU$ denote the periodic complex $K$-theory spectrum.
Bott periodicity shows the coefficient ring $\KU_{\ast}\cong\Z[\beta,\beta^{-1}]$.
Under the isomorphism $\KU_{2}\cong\widetilde{\KU}^{0}(S^{2})$, 
the Bott element $\beta$ maps to the difference $1-\xi_{\infty_{\vert\CP^{1}}}$ between the trivial line bundle 
and the restriction of the tautological line bundle $\xi_{\infty}$ over $\CP^{\infty}$ to $\CP^{1}\cong S^{2}$.
Thus the class $\beta^{-1}(1-\xi_{\infty})$ in $\widetilde{\KU}^{2}(\CP^{\infty})$ defines an orientation on $\KU$.
It is straightforward to show that the corresponding degree $-2$ multiplicative formal group law is $F_{\KU}(\x,\y)=\x+\y-\beta\x\y$, 
and the ring map $\MU_{\ast}\rightarrow\KU_{\ast}$ classifying $F_{\KU}$ sends $v_{0}$ to $p$ and $v_{1}$ to $\beta^{p-1}$ 
for every prime number $p$ and invariant prime ideal $(p,v_{0},v_{1},\dots)$ of the complex cobordism ring.
This map turns $\Z[\beta,\beta^{-1}]$ into a Landweber exact graded $\MU_{\ast}$-algebra because the multiplication 
by $p$ map on the integral Laurent polynomial ring in the Bott element is injective with cokernel $\Z/p[\beta,\beta^{-1}]$, 
and the multiplication by $\beta^{p-1}$ map on the latter is clearly an isomorphism.
Having dealt with these well known topological preliminaries, 
we are now ready to turn to motivic homotopy theory over a noetherian base scheme $S$ of finite Krull dimension.
\vspace{0.1in}

Applying the considerations above to the motivic Landweber exact functor theorem announced in \cite[Theorem 8.6]{NSO} and extended to 
nonregular base schemes in Theorem \ref{theorem:mleft}, 
imply there exists a Tate object $\LK$ in the motivic stable homotopy category $\SH(S)$ and an isomorphism of motivic homology theories
\begin{equation*}
\LK_{\ast,\ast}(-)\cong
\MGL_{\ast,\ast}(-)\otimes_{\MU_{\ast}}\Z[\beta,\beta^{-1}].
\end{equation*}
Here $\MGL$ denotes the algebraic cobordism spectrum, 
and the category $\SH(S)_{\caT}$ of Tate objects or cellular spectra refers to the smallest localizing triangulated 
subcategory of $\SH(S)$ containing all mixed motivic spheres \cite{DI}, \cite{NSO}.
In addition, 
there exists a quasi-multiplication $\LK\wedge\LK\rightarrow\LK$, 
turning $\LK$ into a commutative monoid modulo phantom maps,  
which represents the ring structure on the Landweber exact homology theory $\LK_{\ast,\ast}(-)$.
Recall from \cite{VVicm} the motivic spectrum $\KGL$ representing homotopy algebraic $K$-theory on $\SH(S)$. 
We show:
\begin{proposition}
\label{proposition:LKKGL}
There is an isomorphism in $\SH(S)$ between $\LK$ and $\KGL$.
\end{proposition}
\begin{proof}
It suffices to prove the result when $S=\Spec(\Z)$ because $\LK$ and $\KGL$ pull back to general base schemes; 
the former by \cite[Proposition 8.4]{NSO} and the latter by the following argument, cf.~\cite[\S6.2]{VVicm}:
To compute the derived pullback one takes cofibrant replacements for a model structure such that the 
pullback and pushforward functors form a Quillen pair; 
this fails for the injective motivic model structure \cite{MV}, but works for the projective motivic model structure \cite{DRO}.
Recall that $\Z\times\BGL$ comprises all the constituent motivic spaces in $\KGL$.
Let $f\colon Q\KGL\rightarrow\KGL$ be a cofibrant functorial replacement of $\KGL$ so that $f$ is a trivial fibration.
Now trivial fibrations in the stable model structure for motivic spectra coincide with the trivial fibrations in the 
projective model structure on motivic spectra, 
so it follows that this is a level weak equivalence. 
This model categorical argument shows that a 
levelwise derived pullback is a derived pullback of motivic spectra. 
Since the spaces $\BGL$ pull back this implies the pullback gives something level equivalent to the standard model for $\KGL_S$.

The motivic Bott element turns $\KGL$ into an oriented motivic ring spectrum with multiplicative formal group law \cite[Example 2.2]{SO}.
It follows that there is a canonical transformation $\LK_{**}(-)\rightarrow\KGL_{**}(-)$ of $\MGL_{**}(-)$-module homology theories.
Since $\SH(\Z)_\caT$ is a Brown category \cite[Lemma 8.2]{NSO}, 
we conclude there is a map $\Phi\colon\LK\to\KGL$ of $\MGL$-module spectra representing the transformation.
\vspace{0.1in}

Now since $\LK$ and $\KGL$ are Tate objects; the former by construction and the latter by \cite{DI}, 
$\Phi$ is an isomorphism provided the naturally induced map between motivic stable homotopy groups
\begin{equation}
\label{mapinhomotopy}
\xymatrix{
\pi_{\ast,\ast}\Phi\colon\pi_{\ast,\ast}(\LK\ar[r] & \KGL) }
\end{equation}
is an isomorphism \cite[Corollary 7.2]{DI}.
The map in (\ref{mapinhomotopy}) is a retract of
\begin{equation}
\label{mapinMGLhomology}
\xymatrix{
\MGL_{\ast,\ast}\Phi\colon\MGL_{\ast,\ast}(\LK\ar[r] & \KGL). }
\end{equation}
To wit, 
there is a commutative diagram
\[ 
\xymatrix{ 
\LK\ar[r]\ar[d]^\Phi & \MGL\wedge\LK\ar[r]\ar[d]^{\MGL\wedge\Phi} & \LK\ar[d]^\Phi\\
\KGL\ar[r] & \MGL\wedge\KGL\ar[r] & \KGL }
\]
where the horizontal compositions are the respective identity maps.
Thus it suffices to prove the map in (\ref{mapinMGLhomology}) is an isomorphism.
\vspace{0.1in}

By \cite[Remark 9.2]{NSO} there is an isomorphism
\begin{equation}
\label{equation:firstisomorphism}
\MGL_{**}\LK\cong
\MGL_{**}\otimes_{\MU_*} \MU_*\KU.
\end{equation}
The latter identifies with the target of $\MGL_{\ast,\ast}\Phi$ via the isomorphisms
\begin{equation}
\label{equation:seconfisomorphism}
\begin{aligned}
\MGL_{**}\KGL 
& \cong \colim_{n}\,\MGL_{*+2n,*+n}(\Z\times\BGL)\\
& \cong \MGL_{**}\otimes_{\MU_*}\colim_n \,\MU_{*+2n}(\Z\times\BU) \\
& \cong \MGL_{**}\otimes_{\MU_*}\MU_*\KU. 
\end{aligned}
\end{equation}
In the third isomorphism we use compatibility of the structure maps of $\KGL$ and $\KU$, 
and the computation $\MGL_{**}(\BGL)\cong\MGL_{**}\otimes_{\MU_*}\MU_*\BU$.

It remains to remark that $\MGL_{**} \Phi$ is the identity with respect to the isomorphisms in 
(\ref{equation:firstisomorphism}) and (\ref{equation:seconfisomorphism}).
\end{proof}

\begin{remark}
Proposition \ref{proposition:LKKGL} reproves the general form of the Conner-Floyd theorem for homotopy algebraic $K$-theory \cite{SO}. 
We would like to point out that this type of result is more subtle than in topology:
Recall that if $E$ is a classical homotopy commutative ring spectrum which is complex orientable and the resulting formal 
group law is Landweber exact, 
then $E_*(-) = \MU_*(-) \otimes_{\MU_*} E_*$.
Now suppose $S$ has a complex point and let $\EE$ be the image of $\HQ$ under the right adjoint of the realization functor.
It is easy to see that $\EE$ is an orientable motivic ring spectrum with additive formal group law and $\EE_* \cong \LQ_*$,
at least for fields.
However, 
since $\EE_{p,q} = \EE_{p,q'}$ for all $p,q,q'\in\Z$, 
we get that $\EE \not\cong \LQ$.
\end{remark}

In \cite{SO} it is shown that the Bott inverted infinite projective space is the universal multiplicative 
oriented homology theory.
Thus there exists an isomorphism between the representing spectra $\LK$ and $\Sigma^\infty\PP^\infty_+[\beta^{-1}]$. 
By \cite[Remark 9.8(ii)]{NSO} there is a unique such isomorphism,
and it respects the monoidal structure.
This allows us to conclude there is an isomorphism in $\SH(S)$ between $\Sigma^\infty\PP^\infty_+[\beta^{-1}]$ and $\KGL$.
An alternate computational proof of this isomorphism was given in \cite{SO} and another proof was announced in \cite{GS}.
In the event of a complex point on $S$,
the topological realization functor maps $\LK_{\ast,\ast}(-)$ to $\KU_{\ast}(-)$ and likewise for $\KGL$.
For fun,
we note that running the exact same tape in stable homotopy theory yields Snaith's isomorphism 
$\Sigma^\infty\CP^\infty_+[\beta^{-1}]\cong\KU$.
In the topological setting, 
the argument above could have been given over thirty years ago.
\vspace{0.1in}

In the proof of Proposition \ref{proposition:LKKGL} we used the universal property of the homology theory associated to $\MGL$.
For the proof this is only needed over $\Z$.
In the following sections we shall justify the claim that this holds over general base schemes.

\begin{theorem}
\label{theorem:MGLuniversality}
Suppose $S$ is a noetherian scheme of finite Krull dimension and $\EE$ a commutative ring spectrum in $\SH(S)$.
Then there is a bijection between the sets of ring spectra maps $\MGL\rightarrow\EE$ in $\SH(S)$ and the orientations on $\EE$.
\end{theorem}
The proof also shows there is an analogous result for homology theories.
For fields,
this result is known thanks to the works of Vezzosi \cite{Vezzosi} and more recently of Panin-Pimenov-R{\"o}ndigs \cite{PPR}.
The generalization of their results to regular base schemes is straightforward using, 
for example, 
the results in \cite{NSO}.

In order to prove Theorem \ref{theorem:MGLuniversality} we compute the $\EE$-cohomology of $\MGL$ for $\EE$ an oriented motivic 
ring spectrum;
in turn, 
this makes use of the Thom isomorphism for universal vector bundles over Grassmannians.
Usually the base scheme is tactically assumed to be regular for the construction of Chern classes of vector bundles.
However, 
since the constituent spaces comprising the spectrum $\MGL$ are defined over the integers $\Z$ one gets Chern classes and the 
Thom isomorphism by pulling back classifying maps of line bundles.
In any event, 
one purpose of what follows is to dispense with the regularity assumption for the construction of Chern classes.
For the remaining part of the proof we note that the streamlined presentation in \cite{PPR} following the route layed out in 
algebraic topology carries over verbatim.

\section{Chern classes and oriented motivic ring spectra}
\label{section:Chernclassesandorientedmotivicringspectra}
The basic input required for the theory of Chern classes is that of the first Chern class of a line bundle.
If the base scheme $S$ is regular, there is an isomorphism
$$
\Pic(X)\cong
\Hom_{\H(S)}(X,\PP^\infty)
$$
between the Picard group of $X$ and maps from $X$ to the infinite projective space $\PP^\infty$ in the unstable motivic homotopy 
category $\H(S)$ \cite[Proposition 4.3.8]{MV}.
If $S$ is nonregular, 
there is no such isomorphism because the functor $\Pic(-)$ is not $\A^1$-invariant. 
The first goal of this section is to prove that the most naive guess as to what happens in full generality holds true, 
namely: 
\begin{center}
$\PP^\infty$ represents the $\A^1$-localization of $\Pic$.
\end{center}
\vspace{0.1in}

For $\tau$ some topology on the category of smooth $S$-schemes $\Sm/S$, 
let $\sPr(\Sm/S)^\tau$ denote any one of the standard model structure on simplicial presheaves with $\tau$-local weak equivalences and
likewise for $\sPr(\Sm/S)^{\tau,\A^1}$ and $\tau-\A^1$-local weak equivalences.
The topologies of interest in what follows are the Zariski and Nisnevich topologies, denoted by $\Zar$ and $\Nis$ respectively.
Let $\H(S)$ denote the unstable motivic homotopy category of $S$ associated to the model structure $\sPr(\Sm/S)^{\Nis,\A^1}$.
Throughout we will use the notation $\G$ for the multiplicative group scheme over $S$, 
$\uPic(X)$ for the Picard groupoid of a scheme $X$, 
and $\nu \uPic$ for the simplicial presheaf obtained from (a strictification of) $\uPic$ by applying the nerve functor.

\begin{proposition}
\label{proposition:picpinfty}
\begin{itemize}
\item[(i)] There exists a commutative diagram
$$
\xymatrix{
\cc_X \colon\Pic(X)\ar[rr] & & \Hom_{\H(S)}(X,\PP^\infty) \\
& \Hom_{\Set^{(\Sm/S)^\op}}(X,\PP^\infty) 
\ar[lu]^f \ar[ru]_g &}
$$
which is natural in $X\in\Sm/S$, where
$f$ is given by pulling back the tautological
line bundle on $\PP^\infty$ and $g$ is
the canonical map.
The transformation $\cc$ is an isomorphism when $S$ is regular.
\item[(ii)] There is a natural weak equivalence 
$$
\PP^\infty\simeq\nu\uPic
$$ 
in $\sPr(\Sm/S)^{\Zar,\A^1}$.
\end{itemize}
\end{proposition}

To prepare for the proof of Proposition \ref{proposition:picpinfty}, 
suppose $\GGG\in\sPr(\Sm/S)$ is a group object acting on a simplicial presheaf $X$. 
Let $X/^h \GGG$ denote the associated homotopy quotient. 
It is defined as the total object (homotopy colimit) of the simplicial object:
$$
\xymatrix{
X & X \times \GGG \ar@<.7ex>[l] \ar@<-.7ex>[l] & 
X \times \GGG^2 \ar[l] \ar@<1ex>[l] \ar@<-1ex>[l] & 
\cdots \ar@<1.1ex>[l]_(.4){:} \ar@<-1.1ex>[l] }
$$
Pushing this out to any of the localizations yields local homotopy quotients since the localization is a left adjoint functor, 
and hence preserves (homotopy) colimits. 
If $X \to Y$ is a $\GGG$-equivariant map with $Y$ having trivial $\GGG$-action, 
there is a naturally induced map $X/^h\GGG\to Y$ in the corresponding homotopy category.
\begin{lemma}
\label{lemma:affine-projective}
The naturally induced map 
$$
\xymatrix{
(\A^{n+1}\smallsetminus \{0\}) /^h \G \ar[r] & \PP^n }
$$ 
is a Zariski local weak equivalence.
\end{lemma}
\begin{proof}
Let $\pi:\A^{n+1}\smallsetminus\{0\}\to\PP^n$ be the canonical map, 
fix $0\le i\leq n$ and consider the standard affine open $U_i:=\{ [x_0:\ldots:x_n]\, |\, x_i\neq 0\}\subseteq\PP^n$ and its pullback along $\pi$, 
i.e.~$V_i:=\pi^{-1}(U_i)\subseteq \A^{n+1}\smallsetminus\{0\}$.
Working Zariski locally, 
it suffices to see that $\pi$ induces a weak equivalence $V_i/^h \G\to U_i$. 
There is an $\G$-equivariant isomorphism $V_i\cong\G\times\A^n$ with $\G$ acting trivially on $\A^n$ and by multiplication on $\G$.
It remains to remark that $(\G\times\A^n)/^h \G\cong \A^n$ because the simplicial diagram defining $(\G\times\A^n)/^h \G$ admits an evident retraction.
\end{proof}

By lemma \ref{lemma:affine-projective} there is a weak equivalence $(\A^\infty \smallsetminus \{0\}) /^h \G\to \PP^\infty$ in $\sPr(\Sm/S)^\Zar$.
Let $\BG$ denote the quotient $\mathrm{pt} /^h\G$ in any of the localizations of $\sPr(\Sm/S)$.
\begin{lemma}
\label{lemma:affine-bgm}
The map 
$$\xymatrix{ 
(\A^\infty \smallsetminus \{0\}) /^h \G\ar[r] & \BG }
$$ 
induced by $\A^\infty \smallsetminus \{0\} \to \mathrm{pt}$ is a weak equivalence in $\sPr(\Sm/S)^{\A^1}$.
\end{lemma}
\begin{proof}
The map $\A^\infty \smallsetminus \{0\} \to \mathrm{pt}$ is an equivalence in $\sPr(\Sm / S)^{\A^1}$:
In effect, 
the inclusions $\A^n \smallsetminus \{0\} \hookrightarrow\A^{n+1} \smallsetminus \{0\}$ are homotopic to a constant map by the elementary 
$\A^1$-homotopy 
$$
\xymatrix{
(t,(x_1,\ldots,x_n)) \ar@{|->}[r] & ((1-t)x_1,\ldots,(1-t)x_n,t). }
$$
A finite generation consideration finishes the proof.
See \cite[Proposition 4.2.3]{MV} for a generalization.
\end{proof}
\begin{remark}
We thank the referee for suggesting that by inserting $t$ in the first coordinate gives a homotopy which is compatible with the 
inclusions $i_n$ of $\A^n \smallsetminus \{0\}$ into $\A^{n+1} \smallsetminus \{0\}$ sending $(x_1,\ldots,x_n)$ to $(x_1,\ldots x_n,0)$
and a contraction of $\A^\infty \smallsetminus \{0\}$ onto the point $(1,0,\dots)$.
To wit, 
the homotopies $H_n:\A^1\times \A^n-\{0\}\to\A^{n+1}-\{0\}, (t,x_1,\ldots,x_n)\mapsto (t,(1-t)x_1,\ldots,(1-t)x_n))$ make the diagrams
\begin{equation*}
\xymatrix{ 
\A^1\times \A^{n}-\{0\} \ar[r]^-{H_n}\ar[d]_{\id\times i_n} & \A^{n+1}-\{0\} \ar[d]^{i_{n+1}}\\
\A^1\times \A^{n+1}-\{0\}\ar[r]^-{H_{n+1}} & \A^{n+2}-\{0\} }
\end{equation*}
commute.
Thus by considering  $\A^{\infty}-\{0\}$ as a presheaf we get a homotopy
\[ 
H:=\mathrm{colim } H_n:\A^1\times\A^{\infty}-\{0\}\to \A^{\infty}-\{0\}.
\]
Note that $H(1,x_1,\ldots)=(1,0,0,\ldots)$ and $H(0,x_1,\ldots)=(0,x_1,x_2,\ldots)$, 
i.e.~this is a homotopy from some constant map on $\A^{\infty}-\{0\}$ to the ``shift map.''
The latter is homotopic to the identity map on $\A^{\infty}-\{0\}$ via the homotopy
\[ 
H'\colon \A^1\times\A^{\infty}-\{0\}\to\A^{\infty}-\{0\}; 
(t,x_1,\ldots)\mapsto ((1-t)x_1,x_2+t(x_1-x_2),\ldots, x_n+t(x_{n-1}-x_n),\ldots).
\]
Note that this map does not pass through the origin.
Moreover, 
$H'(0,-)$ is the identity map and $H'(1,-)$ is the ``shift map.''
\end{remark}

\begin{corollary}
\label{corollary:bgm-projective}
There is a natural weak equivalence 
$$
\BG \simeq \PP^\infty
$$ 
in $\sPr(\Sm/S)^{\Zar,\A^1}$.
\end{corollary}
\begin{proof}
Combine the equivalences $(\A^\infty \smallsetminus \{0\}) /^h \G \to \PP^\infty$ and $(\A^\infty \smallsetminus \{0\}) /^h \G \to \BG$ 
from Lemmas \ref{lemma:affine-projective} and \ref{lemma:affine-bgm}.
\end{proof}

\begin{lemma}
\label{lemma:bgm-pic}
There is a natural weak equivalence 
$$
\BG \simeq \nu \uPic
$$ 
in $\sPr(\Sm/S)^\Zar$.
\end{lemma}
\begin{proof}
Denote by $\uPic' \subset \uPic$ the subpresheaf of groupoids consisting objectwise of the single object formed by the trivial line bundle. 
Then $\nu \uPic'$ provides a model for $\BG$ in $\sPr(\Sm/S)$. 
By considering stalks it follows that $\nu \uPic' \to \nu \uPic$ is a Zariski local weak equivalence. 
\end{proof}

The fact that $\uPic$ is a stack in the flat topology which is an instance of faithfully flat descent (for projective rank one modules) 
implies  
\begin{equation}
\label{descentofpic}
\nu \uPic\mbox{ satisfies Zariski descent.}
\end{equation}
Indeed, 
$\nu$ is a right Quillen functor for the local model structure on groupoid valued presheaves, 
so it preserves fibrant objects.
Note that $\nu \uPic$ satisfies descent in any topology which is coarser than the flat one.
In general, 
$\nu \uPic$ is not $\A^1$-local.

\begin{proof}[Proof (of Proposition \ref{proposition:picpinfty})]
Combining Corollary \ref{corollary:bgm-projective} and Lemma \ref{lemma:bgm-pic} yields $(ii)$.
We define the natural transformation $\cc$ in part $(i)$ using the composition
\begin{align*}
\Pic(X) 
& =\pi_0\nu\uPic(X)\\
& \stackrel{(\ref{descentofpic})}{\cong} \Hom_{\Ho(\sPr(\Sm/S)^\Zar)}(X,\BG)\\
& \to \Hom_{\H(S)}(X,\BG) \\ 
& \cong \Hom_{\H(S)}(X,\PP^\infty).
\end{align*}
If $S$ is regular, 
then $\nu\uPic$ is $\A^1$-local, 
and hence the composition is an isomorphism.

The triangle in $(i)$ commutes since the map $\PP^\infty \to \BG$ used to construct the transformation $\cc$ 
classifies the tautological line bundle on $\PP^\infty$.
\end{proof}

Let $\Th(\mathcal{T}(1))$ denote the Thom space of the tautological vector bundle $\mathcal{T}(1)\equiv\mathcal{O}_{\PP^{\infty}}(-1)$ 
with fiber $\AA^{1}$ over the base point $S\hookrightarrow\PP^{\infty}$.
We recall the notion of oriented motivic ring spectra formulated in \cite{PPR}, 
cf.~\cite{HK}, \cite{Morel:basicproperties} and \cite{Vezzosi}.
The unit of a commutative,
i.e.~a commutative monoid in $\SH(S)$, 
$\PP^{1}$-ring spectrum $\EE$ defines $1\in\EE^{0,0}(S_+)$.
Applying the $\PP^{1}$-suspension isomorphism to $1$ yields the element $\Sigma_{\PP^{1}}(1)\in\EE^{2,1}(\PP^{1},\infty)$.
The canonical covering of $\PP^{1}$ defines motivic weak equivalences
\[ \xymatrix{\PP^{1} \ar[r]^-\sim & \PP^{1}/\AA^{1} &
\TT \equiv \AA^{1}/\AA^{1}\smallsetminus \lbrace 0 \rbrace \ar[l]_-\sim} \]
of pointed motivic spaces inducing isomorphisms
$\EE^{**}(\PP^{1},\infty) \leftarrow \EE^{**}(\AA^{1}/\AA^{1}\smallsetminus\{0\})\rightarrow\EE^{**}(\TT)$.
Let $\Sigma_{\TT}(1)$ be the image of $\Sigma_{\PP^{1}}(1)$ in $\EE^{2,1}(\TT)$.
\begin{definition}
\label{OrientationViaThom}
Let $\EE$ be a commutative $\PP^{1}$-ring spectrum.
\begin{enumerate}[(i)]
\item
A Chern orientation on $\EE$ is a class $\chclass\in\EE^{2,1}(\PP^{\infty})$ such that $\chclass|_{\PP^{1}}= - \Sigma_{\PP^{1}}(1)$.
\item
A Thom orientation on $\EE$ is a class $\thclass \in \EE^{2,1}(\Th(\mathcal T(1))$ such that its restriction to the Thom space 
of the fiber over the base point coincides with $\Sigma_{\TT}(1)\in\EE^{2,1}(\TT)$.
\item
An orientation $\sigma$ on $\EE$ is either a Chern orientation or a Thom orientation.
A Chern orientation $\chclass$ coincides with a Thom orientation $\thclass$ if $\chclass=z^*(\thclass)$. 
\end{enumerate}
\end{definition}
\begin{example}
Let $u_1\colon\Sigma^{\infty}_{\PP^{1}}(\Th(\mathcal{T}(1)))(-1) \to \MGL$ be the canonical map of $\PP^{1}$-spectra. 
The $-1$ shift refers to $\PP^{1}$-loops.
Set $\thclass^{\MGL}\equiv u_1\in \MGL^{2,1}(\Th(\mathcal{T}(1)))$.
Since $\thclass^{\MGL}|_{\Th(\mathbf{1})}= \Sigma_{\PP^{1}}(1)\in\MGL^{2,1}(\Th(\mathbf{1}))$,
the class $\thclass^{\MGL}$ is an orientation on $\MGL$.
\end{example}

Using Proposition \ref{proposition:picpinfty} we are now ready to define the first Chern classes.
\begin{definition}
Suppose $(\EE,\sigma)$ is an oriented motivic ring spectrum in $\SH(S)$, 
$X\in\Sm/S$ and $\caL$ a line bundle on $X$. 
Then the first Chern class of $\caL$ for the given orientation is defined as $\cc_{\caL}\equiv\cc_X(\caL)^*(\sigma)\in\EE^{2,1}(X)$.
\end{definition}

It is clear that this definition recovers the previously considered construction in case $S$ is regular.

The following key result and its proof due to Morel \cite{Morel:basicproperties} are included for completeness.
\begin{lemma}
\label{lemma:morel}
For every $n\ge 1$ there is a canonical weak equivalence $\PP^n/\PP^{n-1} \cong (\PP^1)^{\wedge n}$ such that the naturally induced diagram
\begin{equation*}
\xymatrix{ 
\PP^{n}\ar[d]\ar[r]^-{\Delta^{\wedge n}} & (\PP^{n})^{\wedge n} \\
\PP^{n}/\PP^{n-1}\ar[r]^-{\cong} & (\PP^{1})^{\wedge n}\ar[u] } 
\end{equation*}
commutes in $\H(S)$.
\end{lemma}
\begin{proof}
For $0\leq i\leq n$,
let $U_{i}$ be the open affine subscheme of $\PP^{n}$ of points $[x_0,\dots,x_n]$ such that $x_{i}\neq 0$, 
set $\Omega_i\equiv(\PP^n)^{i-1}\times U_{i}\times(\PP^n)^{n-i}$ and $\Omega\equiv\cup_{1\leq i\leq n}\Omega_{i}
\subseteq (\PP^n)^n$.
The projection $(\PP^{n})^{n}\rightarrow (\PP^{n})^{n}/\Omega$ induces a motivic weak equivalence
$(\PP^{n})^{\wedge n}\rightarrow (\PP^{n}/U_{1})\wedge\dots\wedge (\PP^{n}/U_{n})$.
It allows to replace $(\PP^{n})^{\wedge n}$ by the weakly equivalent smash product of the motivic spaces 
$\PP^{n}/U_{i}$ for $1\leq i\leq n$.
Note that $\PP^{n}\rightarrow (\PP^{n}/U_{1})\wedge\dots\wedge (\PP^{n}/U_{n})$ factors through 
$\PP^{n}/\cup_{1\leq i\leq n}U_{i}$ and the inclusion $\PP^{n-1}\subseteq\cup_{1\leq i\leq n}U_{i}$ is the zero 
section of the canonical line bundle over $\PP^{n-1}$.
Hence the latter map is a motivic weak equivalence and $\PP^{n}\rightarrow\PP^{n}/\cup_{1\leq i\leq n}U_{i}$
induces a motivic weak equivalence $\PP^{n}/\PP^{n-1}\rightarrow\PP^{n}/\cup_{1\leq i\leq n}U_{i}$.
The inclusion $U_{0}\subseteq\PP^{n}$ induces an isomorphism of pointed motivic spaces 
$\AA^{n}/\AA^{n}\smallsetminus\{0\}\cong\PP^{n}/\cup_{1\leq i\leq n}U_{i}$.
Since there are canonical isomorphisms 
$\AA^{n}/\AA^{n}\smallsetminus\{0\}\cong (\AA^{1}/\AA^{1}\smallsetminus\{0\})^{\wedge n}$ and 
$\AA^{1}/\AA^{1}\smallsetminus\{0\}\cong \PP^{1}/\AA^{1}$ there is an induced map
$(\PP^{1}/\AA^{1})^{\wedge n}\rightarrow (\PP^{n}/U_{1})\wedge\dots\wedge (\PP^{n}/U_{n})$ weakly equivalent to
$(\PP^{1})^{\wedge n}\rightarrow (\PP^{n}/U_{1})\wedge\dots\wedge (\PP^{n}/U_{n})$.
\end{proof}

Next we state and prove the projective bundle theorem for oriented motivic ring spectra, 
cf.~the computation for Grassmannians in \cite[Proposition 6.1]{NSO}.
\begin{theorem}
\label{theorem:Ecohomologyofprojectivebundles}
Suppose $\EE$ is an oriented motivic ring spectrum,
$\xi\colon Y\rightarrow X$ a rank $n+1$ vector bundle over a smooth $S$-scheme $X$,
$\PP(\xi)\colon \PP(Y)\rightarrow X$ its projectivization and $\cc_\xi\in\EE^{2,1}(\PP(\xi))$
the first Chern class of the tautological line bundle.
\begin{itemize}
\item[(i)]  Then $\EE^{\ast,\ast}(\PP(\xi))$ is a free $\EE^{*,*}(X)$-module on generators $1=\cc_\xi^0,\cc_\xi,\ldots,\cc_\xi^n$.
\item[(ii)] If $\xi$ is trivial, then there is an isomorphism of $\EE^{*,*}(X)$-algebras
$$
\EE^{\ast,\ast}(\PP(\xi))\cong
\EE^{*,*}(X)[\cc_\xi]/(\cc_\xi^{n+1}).
$$
\end{itemize}
\end{theorem}
\begin{proof}
The claimed isomorphism is given explicitly by 
\begin{equation*}
\xymatrix{
\bigoplus_{i=0}^{n}
\EE^{\ast-2i,\ast-i}(X)\ar[r]^-{\cong} & \EE^{\ast,\ast}(\PP(\xi)); \\
(\x_{0},\dots,\x_{n})\ar@{|->}[r] & \Sigma_{i=0}^{n} (\PP(\xi))^{\ast}(\x_{i})\cc_{\xi}^{i}. }
\end{equation*}
Combining the choice of a trivialization of the vector bundle $\xi$ with a Mayer-Vietoris induction argument shows that $(i)$ follows from $(ii)$. 
Part $(ii)$ is proved by induction on the rank of $\xi$.

Assume the map is an isomorphism for the trivial vector bundle $Y\rightarrow X$ of rank $n$,
and form $Y\oplus\mathcal{O}_{X}\rightarrow X$ of rank $n+1$.
Next form the cofiber sequence 
$\PP(Y)\rightarrow\PP(Y\oplus\mathcal{O}_{X})\rightarrow (\TT_{X}\equiv\A^{1}_{X}/\A^{1}_{X}\smallsetminus\{0\})^{\wedge\, n}\wedge X_{+}$ 
and consider the induced diagram in $\EE$-cohomology:
\begin{equation*}
\xymatrix{
\dots\ar[r] &
\EE^{\ast-2n,\ast-n}(X_{+})\ar[r]\ar[d] &
\bigoplus_{i=0}^{n}\EE^{\ast-2i,\ast-i}(X_{+})\ar[r]\ar[d] &
\bigoplus_{i=0}^{n-1}\EE^{\ast-2i,\ast-i}(X_{+})\ar[r]\ar[d] &
\dots\\
\dots\ar[r] &
\EE^{\ast,\ast}((\TT_{X})^{\wedge n}\wedge X_{+})\ar[r] &
\EE^{\ast,\ast}(\PP(Y\oplus\mathcal{O}_{X}))\ar[r] &
\EE^{\ast,\ast}(\PP(Y))\ar[r] &
\dots }
\end{equation*}

The right hand square commutes using functoriality of first Chern classes and the fact that the restriction map is a ring map. 
The commutativity of the left hand square is exactly the explicit computation of the Gysin map and is where Lemma \ref{lemma:morel} 
enters the proof.

The left hand map is an isomorphism since smashing with $\TT_{X}$ is an isomorphism in $\SH(S)$.
Now use the induction hypothesis and apply the $5$-lemma which proves part $(i)$ in this special case. 
The relation $\cc_\xi^{n+1}=0$ follows using exactness of the lower row of the diagram.
\end{proof}

Theorem \ref{theorem:Ecohomologyofprojectivebundles} allows to define the higher Chern class $\cc_{i\xi}\in\EE^{2i,i}(X)$ of $\xi$ as the 
unique solution of the equation 
\begin{equation*} 
\cc_{\xi}^{n}=
\cc_{1\xi}\cc_{\xi}^{n-1}-
\cc_{2\xi}\cc_{\xi}^{n-2}+
\cdots+
(-1)^{n-1}\cc_{n\xi}.
\end{equation*}

With the theory of Chern classes in hand, 
one constructs a theory of Thom classes by the usual script.
In particular, 
as for regular base schemes one verifies the Thom isomorphism theorem:
\begin{theorem}
\label{theorem:thomisomorphism}
If $\xi\colon Y\rightarrow X$ is a rank $n$ vector bundle with zero section $z\colon X\rightarrow Y$ over a smooth $S$-scheme $X$,
cup-product with the Thom class of $\xi$ induces an isomorphism 
\begin{equation*}
\xymatrix{
(-\cup\thclass_{\xi})\circ \xi^{\ast}\colon \EE^{\ast,\ast}(X)\ar[r]^-{\cong} 
& \EE^{\ast+2n,\ast+n}(Y/(Y\smallsetminus z(X))). }
\end{equation*}
\end{theorem}

Having dealt with these preparations, 
the $\EE$-cohomology computation of Grassmannians in \cite{NSO} extends to nonregular base schemes.
The other parts involved in the proof of the motivic Landweber exact functor theorem given in \cite{NSO} do not require a regular base scheme.
In conclusion,
the following holds (for additional prerequisites we refer to \cite{NSO}): 
\begin{theorem}
\label{theorem:mleft}
Let $S$ be a noetherian base scheme of finite Krull dimension and $\MMM_{*}$ an Adams graded Landweber exact $\MU_{*}$-module.
Then there exists a motivic spectrum $\EE$ in $\SH(S)_{\mathcal{T}}$ and a natural isomorphism
\begin{equation*}
\EE_{\ast\ast}(-)\cong
\MGL_{**}(-)\otimes_{\MU_{*}}\MMM_{*}
\end{equation*}
of homology theories on $\SH(S)$.
\end{theorem}

Moreover,
based on the results above,
the following computations stated for fields in \cite{PPR} hold for $S$.
Let $\Gr$ denote the infinite Grassmannian over $S$.
\begin{theorem}
\label{structure-e_**mgl}
Suppose $\EE$ is an oriented motivic ring spectrum. 
\begin{enumerate}[(i)]
\item 
The canonical map
\begin{equation}
\label{equation:EEtoinverselimit}
\xymatrix{
\EE^{\ast,\ast}(\MGL)\ar[r] & 
\varprojlim\EE^{\ast+2n,\ast+n}(\Th(\mathcal T(n)))=
\EE^{\ast,\ast}[[\cc_1,\cc_2,\cc_3,\dots]] }
\end{equation}
is an isomorphism.
(Here $\cc_i$ is the $i$th Chern class.)
\item 
The canonical map
\begin{equation}
\label{equation:EEMGLMGLtoinverselimit}
\xymatrix{
\EE^{\ast,\ast}(\MGL\wedge\MGL) \ar[r] & 
\varprojlim \EE^{\ast+2n,\ast+n}(\Th(\mathcal T(n))\wedge\Th(\mathcal T(n)))=
\EE^{\ast,\ast}[[\cc^{\prime}_1,\cc^{\prime\prime}_1,\cc^{\prime}_2,\cc^{\prime\prime}_2, \dots]] }
\end{equation}
is an isomorphism.
Here $c^{\prime}_i$ is the $i$th Chern class obtained from projection on the first factor of $\Gr\times\Gr$, 
and likewise $c^{\prime\prime}_i$ is obtained from the second factor.
\end{enumerate}
\end{theorem}

\section{The universal property of $\MGL$}
In this section we observe, 
with no claims of originality other than for the cohomology computations in Theorem \ref{structure-e_**mgl} and the Thom isomorphism 
Theorem \ref{theorem:thomisomorphism},
that the proof of the universal property of $\MGL$ given by Panin-Pimenov-R{\"o}ndigs in \cite{PPR} goes through for noetherian 
base schemes of finite Krull dimension.

\begin{theorem}
\label{theorem:UniversalityTheorem}
Suppose $\EE$ is a commutative $\PP^{1}$-ring spectrum.
Then the assignment 
$\varphi\mapsto\varphi(th^{\MGL})\in\EE^{2,1}(\Th(\mathcal T(1)))$
identifies the set of monoid maps $\varphi\colon \MGL \to \EE$
in $\SH(S)$ with the set of orientations on $\EE$.
Its inverse sends $\thclass\in\EE^{2,1}(\Th(\mathcal T(1)))$ to the unique map
$\varphi \in \EE^{0,0}(\MGL)=\Hom_{\SH(S)}(\MGL,\EE)$ such that 
$u_i^\ast(\varphi)=\thclass(\mathcal T(i))$ in $\EE^{2i,i}(\Th(\mathcal T(i)))$,
where $u_i$ denotes the canonical map $u_i\colon \Sigma^{\infty}_{\PP^{1}}(\Th(\mathcal T(i)))(-i) \to \MGL$.
\end{theorem}
\begin{proof}
Note that $\thclass:=\varphi(\thclass^{\MGL})$ is an orientation on $\EE$ since
\[ 
\varphi(\thclass)|_{Th(\bf 1)}= \varphi(\thclass|_{Th(\bf 1)})=
\varphi(\Sigma_{\PP^{1}}(1))= \Sigma_{\PP^{1}}(\varphi(1))= \Sigma_{\PP^{1}}(1).
\]
Conversely, 
note that $\thclass^\EE$ gives rise to a unique map $\varphi\colon \MGL \to \EE$ in $\SH(S)$ as desired:
In effect, 
the elements $\thclass(\mathcal T(i))$ comprise an element in the target of the isomorphism 
$\EE^{\ast,\ast}(\MGL) \to \varprojlim \EE^{\ast+2i,\ast+i}(\Th(\mathcal T(i)))$ in (\ref{equation:EEtoinverselimit}).
This shows uniqueness of $\varphi\in\EE^{0,0}(\MGL)$.
\\ \indent
To check that $\varphi$ respects the multiplicative structure, 
form the diagram:
\[ 
\xymatrix@C=6em{
\Sigma^{\infty}_{\PP^{1}}(\Th(\mathcal T(i)))(-i) \wedge \Sigma^{\infty}_{\PP^{1}}(\Th(\mathcal T(j)))(-j)
\ar[r]^-{\Sigma^{\infty}_{\PP^{1}}(\mu_{i,j})(-i-j)} \ar[d]_-{u_i\wedge u_j} &
\Sigma^{\infty}_{\PP^{1}}(\Th(\mathcal T(i+j)))(-i-j) \ar[d]^-{u_{i+j}} \\
\MGL \wedge \MGL  \ar[r]^-{\mu_{\MGL}} \ar[d]_-{\varphi \wedge \varphi}
& \MGL  \ar[d]^-{\varphi}  \\
\EE \wedge \EE\ar[r]^-{\mu_{\EE}} & \EE} 
\]
The upper square homotopy commutes since
\begin{eqnarray*}
\varphi \circ u_{i+j} \circ \Sigma^{\infty}_{\PP^{1}}(\mu_{i,j})(-i-j) & = &
\mu^\ast_{i,j}(\thclass(\mathcal T(i+j)))=
\thclass(\mu^\ast_{i,j}(\mathcal T(i+j)))=
\thclass(\mathcal T(i) \times \mathcal T(j)) \\ & = &
\thclass(\mathcal T(i)) \times \thclass(\mathcal T(j))=
\mu_{\EE}(\thclass(\mathcal T(i)) \wedge \thclass(\mathcal T(j))) \\ & = &
\mu_{\EE} \circ ((\varphi \circ u_i) \wedge (\varphi \circ u_j)).
\end{eqnarray*}
Combining (\ref{equation:EEMGLMGLtoinverselimit}) with the equality
$\varphi \circ u_{i+i}\circ\Sigma^{\infty}_{\PP^{1}}(\mu_{i,i})(-2i)=\mu_{\EE}\circ ((\varphi\circ u_i)\wedge (\varphi\circ u_i))$
shows that $\mu_{\EE}\circ (\varphi\wedge\varphi)=\varphi\circ\mu_{\MGL}$ in $\SH(S)$.
\vspace{0.1in}

It remains to show the maps constructed above are inverses of each other.
To begin with, 
if $\thclass\in\EE^{2,1}(\Th(\mathcal O(-1)))$, 
then $\varphi\circ u_i=\thclass(\mathcal T_i)$ for all $i$ and the induced orientation $\thclass^{\prime}:=\varphi(\thclass^{\MGL})$
coincides with $\thclass$ since
$\thclass^{\prime}=\varphi(\thclass^{\MGL})=\varphi(u_1)=\varphi\circ u_1=\thclass(\mathcal T(1))=\thclass$.
On the other hand,  
the monoid map $\varphi^{\prime}$ obtained from the orientation $\thclass:=\varphi(\thclass^{\MGL})$ on $\EE$ satisfies 
$u^*_i(\varphi^{\prime})=\thclass(\mathcal T(i))$ for every $i\geq 0$.
To check that $\varphi^{\prime}=\varphi$,
the isomorphism (\ref{equation:EEtoinverselimit}) shows that it suffices to prove that $u^*_i(\varphi^{\prime})=u^*_i(\varphi)$ 
for all $i\geq 0$.
Now since $u^*_i(\varphi^{\prime})=\thclass(\mathcal T_i)$ it remains to prove that $u^*_i(\varphi)=\thclass(\mathcal T(i))$.
In turn this follows if $u_i= \thclass^{\MGL}(\mathcal T(i))$ in $\MGL^{2i,i}(\Th(\mathcal T(i)))$, 
since then $u^\ast_i(\varphi)= \varphi \circ u_i = \varphi(u_i)=\varphi(\thclass^{\MGL}(\mathcal T(i)))=\thclass(\mathcal T(i))$.
We note that the proof of the equality $u_i= \thclass^{\MGL}(\mathcal T(i))$ given in \cite{PPR} carries over to our setting on 
account of the Thom isomorphism Theorem \ref{theorem:thomisomorphism}.
\end{proof}

\bibliographystyle{plain}
\bibliography{nmpblochproceedings}
\vspace{0.5in}

\begin{center}
Fakult{\"a}t f{\"u}r Mathematik, Universit{\"a}t Regensburg, Germany.\\
e-mail: niko.naumann@mathematik.uni-regensburg.de
\end{center}
\begin{center}
Fakult{\"a}t f{\"u}r Mathematik, Universit{\"a}t Regensburg, Germany.\\
e-mail: Markus.Spitzweck@mathematik.uni-regensburg.de
\end{center}
\begin{center}
Department of Mathematics, University of Oslo, Norway.\\
e-mail: paularne@math.uio.no
\end{center}
\end{document}